\documentclass[10pt,bezier]{article}
\usepackage{amsmath,amssymb,amsfonts,graphicx,caption,subcaption}
\usepackage[colorlinks,linkcolor=red,citecolor=blue]{hyperref}

\newtheorem{prethm}{{\bf Theorem}}[section]
\newenvironment{thm}{\begin{prethm}{\hspace{-0.5
em}{\bf.}}}{\end{prethm}}
\newtheorem{prepro}{{\bf Theorem}}

\newtheorem{precor}[prethm]{{\bf Corollary}}
\newenvironment{cor}{\begin{precor}{\hspace{-0.5
em}{\bf.}}}{\end{precor}}
\newtheorem{preconj}[prethm]{{\bf Conjecture}}
\newenvironment{conj}{\begin{preconj}{\hspace{-0.5
em}{\bf.}}}{\end{preconj}}
\newtheorem{preremark}[prethm]{{\bf Remark}}

\newtheorem{prelem}[prethm]{{\bf Lemma}}
\newenvironment{lem}{\begin{prelem}{\hspace{-0.5
em}{\bf.}}}{\end{prelem}}
\newtheorem{preque}[prethm]{{\bf Question}}

\newtheorem{preobserv}[prethm]{{\bf Observation}}
\newenvironment{observ}{\begin{preobserv}{\hspace{-0.5
em}{\bf.}}}{\end{preobserv}}

\newtheorem{preproposition}[prethm]{{\bf Proposition}}

\newtheorem{preproof}{{\bf Proof.}}
\newtheorem{preprooff}{{\bf Proof}}

\newenvironment{proof}[1]{\begin{preproof}{\rm
#1}\hfill{$\Box$}}{\end{preproof}}

\newtheorem{preproofs}{{\bf The second proof of }}

\newenvironment{proofs}[1]{\begin{preproofs}{\rm
#1}\hfill{$\Box$}}{\end{preproofs}}

\newtheorem{preprooft}{{\bf The third proof of }}

\newtheorem{preproofF}{{\bf Proof of}}

\title{\bf\Large 
Spanning tree-connected subgraphs with small degrees 
}
\author{
{\normalsize{\sc Morteza Hasanvand${}$}}
\vspace{3mm}
\\{\footnotesize{${}$\it Department of Mathematical Sciences, Sharif University of Technology, Tehran, Iran}}
{\footnotesize{}}
\\{\footnotesize{ $\mathsf{morteza.hasanvand@alum.sharif.edu }$ }}
\\{\footnotesize{${}$\it Dedicated to Hisako Tsunoda on the occasion of her $67$th birthday}}
}

\date{}
\begin{document}
\maketitle
\begin{abstract}{Let $G$ be a graph with a spanning subgraph $F$, let $m$ be a positive integer, and let $f$ be a positive integer-valued function on $V(G)$. In this paper, we show that if for all $S\subseteq V(G)$, $$\Omega_m(G\setminus S)\le \sum_{v\in S}\big(f(v)-2m\big)+m+\Omega_m(G[S]),$$ then $G$ has a spanning $m$-tree-connected subgraph $H$ containing $F$ such that for each vertex $ v$, $d_H(v)\le f(v)+\max\{0,d_F(v)-m\}$, where $G[S]$ denotes the induced subgraph of $G$ with the vertex set $S$ and $\Omega_m(G_0)$ is a parameter to measure $m$-tree-connectivity of a given graph $G_0$. 

By applying this result, we show that every $k$-edge-connected graph $G$ with $k\ge 2m$ has a spanning $m$-tree-connected subgraph $H$ such that $d_H(v)\le \big\lceil \frac{m}{k}(d_G(v)-2m)\big\rceil+2m$ for each $v\in V(H)$; moreover, if $G$ is $k$-tree-connected and $k\ge m$, then $G$ has a spanning $m$-tree-connected subgraph $H$ such that $d_H(v)\le \big\lceil \frac{m}{k}(d_G(v)-m)\big\rceil+m$ for each $v\in V(H)$. As a consequence, we conclude that every $(r-2m)$-edge-connected graph with $r\ge 4m$ admits a spanning $m$-tree-connected subgraph with maximum degree at most $3m$.

Next, we prove that a graph $G$ admits a spanning $m$-tree-connected subgraph $H$ satisfying $\Delta(H) \le 2m+1$, if for all $S\subseteq V(G)$, $$ \omega(G\setminus S)+\small {\frac{m+1}{2}}\, iso(G\setminus S) \le \frac{1}{m}|S|+1,$$ where $\omega(G\setminus S)$ and $iso(G\setminus S)$ denote the number of components and the number of isolated vertices of $G\setminus S$, respectively. As a consequence, we conclude that every $m(n-1)$-connected $K_{1, n}$-free simple graph with a sufficiently large minimum degree and $n\ge 3$ admits a spanning $m$-tree-connected subgraph with maximum degree at most $2m+1$.
} {\small
}
\end{abstract}
%
%
%
%
%
%
%
%
%
%
%
%
%
%
\section{Introduction}
In this article, graphs have no loops, but multiple edges are allowed and a simple graph is a graph without multiple edges.
 Let $G$ be a graph. 
The vertex set, the edge set, the minimum degree, the maximum degree, and the number of components of $G$ are denoted by $V(G)$, $E(G)$, $\delta(G)$, $\Delta(G)$, and $\omega(G)$, respectively. 
We also denote by $iso(G)$ (resp. $I(G)$) the number (resp. the set) of isolated vertices of $G$.
For a set $X\subseteq V(G)$, we denote by $G[X]$ the induced subgraph of $G$ with the vertex set $X$ containing
precisely those edges 
of $G$ whose ends lie in~$X$.
The degree $d_G(v)$ of a vertex $v$ is the number of edges of $G$ incident to $v$.
We also denote by $d_G(X)$ the number of edges of $G$ with exactly one end in $X$.
For a spanning subgraph $H$ with a given integer-valued function $h$ on $V(H)$,
the {\bf total excess of $H$ from $h$} is defined as follows:
$$te(H, h)=\sum_{v\in V(H)}\max\{0,d_H(v)-h(v)\}.$$
According to this definition, $te(H,h)=0$ if and only if for each vertex $v$, $ d_H(v)\le h(v)$.
A spanning subgraph of a graph $G$ is called a {\bf factor}.
For a set $A$ of integers, an {\bf $A$-factor} is a spanning subgraph with vertex degrees in $A$. 
Let $F$ be a factor of $G$.
For an edge set $E$, we denote by $F-E$ the graph obtained from $F$ by removing the edges of $E$ from $F$.
Likewise, we denote by $F+E$ the graph obtained from $F$ by inserting the edges of $E$ into $F$.
For convenience, we use $e$ instead of $E$ when $E=\{e\}$.
For two edge sets $E_1$ and $E_2$, we also use the notation $E_1+E_2$ for the union of them.
A component of $F$ is said to be {\bf trivial}, if it consists of only one vertex.
 Likewise, $F$ is said to be {\bf trivial}, if it has no edge. 
A vertex set $S$ of a graph $G$ is called {\bf independent}, if there is no edge of $G$ connecting vertices in $S$.
Let $S\subseteq V(G)$. 
 The graph obtained from $G$ by removing all vertices of $S$ is denoted by $G\setminus S$.
Denote by $e_G(S)$ the number of edges of $G$ with both ends in $S$. 
Let $P$ be a partition of $V(G)$.
Denote by $e_G(P)$ the number of edges of $G$ whose ends lie in different parts of $P$.
The graph obtained from $G$ by contracting all vertex sets of $P$ is denoted by $G/P$.
A graph $G$ is called {\bf$m$-tree-connected}, if it has $m$ edge-disjoint spanning trees. 
(Note that the trivial graph having one vertex and no edge is also $m$-tree-connected).
In addition, an $m$-tree-connected graph $G$ is called {\bf minimally $m$-tree-connected},
if for any edge $e$ of $G$, the graph $G- e$ is not $m$-tree-connected.
Hence, an $m$-tree-connected graph is minimally $m$-tree-connected if and only if $|E(G)|=m(|V(G)|-1)$. 
(We should note that somewhere we consider a minimal $m$-tree-connected subgraph of $G$ that contains a fixed vertex subset $X$ of $V(G)$; in this case, our minimality is with respect to inclusion and the desired subgraph is not necessarily minimally $m$-tree-connected).
The vertex set of any graph $G$ can be expressed uniquely as a disjoint union of vertex sets of maximal $m$-tree-connected subgraphs.
(More precisely, we can define the relation on $V (G)$ as follows: $x\sim y$ if $x$ and $y$ are contained in an $m$-tree-connected subgraph of $G$. This is an equivalence relation which partitions $V(G)$ into equivalence classes; see Observation~\ref{observ:definition:tree-connected}). These subgraphs are called the {\bf $m$-tree-connected components} of $G$.
For a graph $G$, we define the parameter $\Omega_m(G)=m|P|-e_G(P)$ to measure tree-connectivity, where $P$ is the unique partition of $V(G)$ obtained from the $m$-tree-connected components of $G$.
Note that $\Omega_1(G)$ is the same number of components of $G$, while $\frac{1}{m}\Omega_m(G)$ is less than or equal to the number of $m$-tree-connected components of $G$.
The definition implies that
the null 
graph $K_0$ with no vertices is not $m$-tree-connected and $\Omega_m(K_0)=0$.
(Note that every nonnull graph $G$ satisfies $\Omega_m(G)\ge m$ and the equality holds if and only if $G$ is $m$-tree-connected; see Theorem~\ref{thm:maximum-partition}).
In this paper, we assume that all graphs are nonnull, except for the graphs that are obtained by removing vertices.
We say that a graph $G$ is {\bf $m^+$-tree-connected}, if it is non-trivial and $G-e$ remains $m$-tree-connected for every edge $e$. 
A graph $F$ is {\bf $m$-sparse}, if $e_F(S)\le m|S|-m$ for all nonempty subsets $S\subseteq V(F)$. Clearly, $1$-sparse graphs are forests.
It is not hard to show that a graph $F$ is $m$-sparse if and only if all $m$-tree-connected components of $F$ are minimally $m$-tree-connected; see Corollary~\ref{cor:spare-component}.
We will show that every $m$-tree-connected graph $H$ with the minimum number of edges containing a given $m$-sparse subgraph $F$ 
 is also minimally $m$-tree-connected.
Let $t$ be a positive real number.
A graph $G$ is said to be {\bf $t$-tough}, if $\omega(G\setminus S)\le \max\{1,\frac{1}{t}|S|\}$ for all $S\subseteq V(G)$. 
Likewise, $G$ is said to be {\bf $m$-strongly $t$-tough},
 if $\frac{1}{m}\Omega_m(G\setminus S)\le \max\{1,\frac{1}{t}|S|\}$ for all $S\subseteq V(G)$.
We will show that every $m$-strongly $t$-tough graph must be $t$-tough; see Theorem~\ref{thm:comparison}. 
In addition, we show that tough enough graphs with sufficiently large
order are also $m$-strongly tough enough; see Corollary~\ref{cor:tough-enough}.
We say that a simple graph $G$ is {\bf $K_{1,f}$-free}, if $G$ does not have an induced star of size $f(v)$ with center $v$ for all $v\in V(G)$, where $f$ is a positive integer-valued function on $V(G)$.
Throughout this article, all variables $k$ and $m$ are positive integers.
%
%
%
%
%
%
%

Recently, the present author~\cite{ ClosedWalks} investigated spanning trees with small degrees and established
the following theorem. This result is an improvement of several results due to Win (1989)~\cite{Win-1989}, Ellingham and Zha~(2000)~\cite{Ellingham-Zha-2000}, and Ellingham, Nam, and Voss~(2002)~\cite{Ellingham-Nam-Voss-2002}.
\begin{thm}{\rm(\cite{ClosedWalks})}\label{thm:c=2:G[S]}
{Let $G$ be a graph with a factor $F$. 
Let $f$ be a positive integer-valued function on $V(G)$. If for all $S\subseteq V(G)$, 
$$\omega(G\setminus S)\le \sum_{v\in S}\big(f(v)-2\big)+1+\omega(G[S]),$$ 
then $G$ has a connected factor $H$ containing $F$ such that for each vertex $v$,
$d_H(v)\le f(v)+\max\{0, d_F(v)-1\}$.
}\end{thm}

We derived the following result from Theorem~\ref{thm:c=2:G[S]} which is an improvement of some results due to 
Liu and Xu (1998)~\cite{MR1621287}, Ellingham, Nam, and Voss~(2002)~\cite{Ellingham-Nam-Voss-2002}, and the present author (2015) \cite{SpanningTreeEulerian-2015}.
\begin{thm}{\rm(\cite{ClosedWalks})}\label{intro:thm:edge-connected}
{Let $G$ be a connected graph. Then every matching of $G$ can be extended to a connected factor $H$ such that for each vertex $v$, 
$$d_H(v)\le 
 \begin{cases}
\lceil \frac{d_G(v)-2}{k}\rceil+2,	&\text{if $G$ is $k$-edge-connected};\\ 
\lceil \frac{d_G(v)-1}{k}\rceil+1,	&\text{if $G$ is $k$-tree-connected}.
\end {cases}$$
}\end{thm}

In this paper, we investigate tree-connected factors with small degrees
 and generalize Theorem~\ref{thm:c=2:G[S]} toward this concept by proving the following theorem.
\begin{thm}\label{intro:thm:sufficient}
Let $G$ be a graph with a factor $F$ and let $f$ be a positive integer-valued function on $V(G)$.
If for all $S\subseteq V(G)$, $$\Omega_m(G\setminus S)\le \sum_{v\in S}\big(f(v)-2m\big)+m+\Omega_m(G[ S]),$$
then $G$ has an $m$-tree-connected factor $H$ containing $F$ such that for each vertex $v$,
 $d_H(v)\le f(v)+\max\{0, d_F(v)-m\}$.
\end{thm}

In Section~\ref{sec:edge-connected-graphs}, we derive the following result from Theorem~\ref{intro:thm:sufficient} on highly edge-connected graphs. As a consequence, we conclude that every $(r-2m)$-edge-connected $r$-regular graph with $r\ge 4m$ must have an $m$-tree-connected factor with maximum degree at most $3m$. 
\begin{thm}
{Let $G$ be a graph. Then every factor of $G$ with maximum degree at most $m$ can be extended to an $m$-tree-connected factor $H$ such that for each vertex $v$, 
$$d_H(v)\le 
 \begin{cases}
\lceil \frac{m(d_G(v)-2m)}{k}\rceil+2m,	&\text{if $G$ is $k$-edge-connected and $k\ge 2m$};\\ 
\lceil \frac{m(d_G(v)-m)}{k}\rceil+m,	&\text{if $G$ is $k$-tree-connected and $k\ge m$}.
\end {cases}$$
}\end{thm}
%
%

In 1973 Chv\'atal~\cite{Chvatal-1973} conjectured that tough enough graphs of order at least three admit a Hamiltonian cycle.
In Section~\ref{sec:tough-graph}, we establish a sufficient toughness-type condition for the existence of tree-connected factors with a bounded maximum degree as the following theorem. 

\begin{thm}\label{intro:thm:tough-m2:verify}
{Let $G$ be a graph. Then every factor of $G$ with maximum degree at most $m$ can be extended to an $m$-tree-connected factor $H$ with $\Delta(H)\le 2m+1$, if for all $S\subseteq V(G)$, 
$$ \omega(G\setminus S)+\small {\frac{m+1}{2}}\, iso(G\setminus S) \le \frac{1}{m}|S|+1.$$
}\end{thm}

In addition, we show that this result is sharp in the sense that the coefficient of $|S|$ cannot be increased according to the following theorem.
\begin{thm}
{Let $m$ be an integer with $m\ge 2$. For every real number $\varepsilon \in (0,1)$, there are infinitely many $2m$-connected graphs $G$ having no $m$-tree-connected factors $H$ with $\Delta(H)\le 2m+1$, while for all $S\subseteq V(G)$ with $|S|\ge 2m$, 
$$ \omega(G\setminus S)+\small {\frac{m+1}{2}}\, iso(G\setminus S) \le (\frac{1}{m}+\varepsilon)|S|.$$
}\end{thm}

It remains to be decided whether 
higher toughness can guarantee the existence
 of an $m$-tree-connected factor with maximum degree at most $2m$.
The special case $m=1$ of this question investigates the existence of Hamiltonian paths in tough enough graphs.
We put forward the following conjecture for this purpose.
\begin{conj}
{For every positive integer $m$, there is a positive real number $t_m$ such that every $t_m$-tough graph $G$ of order at least $2m$ admits an $m$-tree-connected factor $H$ with $\Delta(H)\le 2m$.
}\end{conj}
%
%
%
%
%
%
%
%
%
%
%
%
\section{Basic tools}
\label{subsec:Basic tools}

In this section, 
we present some basic tools for working 
with tree-connected graphs.
We begin 
with the following well-known result which gives
a criterion for a graph to have 
$m$ edge-disjoint spanning trees.
\begin{thm}{\rm (Nash-Williams~\cite{Nash-Williams-1961} and Tutte~\cite{Tutte-1961})}
\label{thm:Nash-Williams,Tutte}
{A graph $G$ is $m$-tree-connected if and only if for every partition $P$ of $V(G)$, $e_G(P)\ge m(|P|-1)$.
}\end{thm}
For every vertex $v$ of a graph $G$, consider an induced $m$-tree-connected subgraph of $G$ containing $v$ with the maximal order.
It is known that these subgraphs are unique and partition the vertex set of $G$; 
see~\cite{Catlin-1989}, and \cite[Lemma 4.7]{Shu-Zhang-Zhang-2012}.
In fact, these subgraphs are the $m$-tree-connected components of G that were introduced in the
Introduction. The following observation simply shows that these subgraphs are well-defined. 
\begin{observ}\label{observ:definition:tree-connected}
{Let $G$ be a graph and let $X\subseteq V(G)$ and $Y\subseteq V(G)$. If $G[X]$ and $G[Y]$ are $m$-tree-connected
 and $X\cap Y\neq \emptyset$, then $G[X\cup Y]$ is also $m$-tree-connected.
}\end{observ}
\begin{proof}
{Let $P$ be a partition of $X\cup Y$. 
Define $P'$ and $P_0$ to be the partitions of $X$ and $Y$ with 
$$P'=\{A\cap X: A\in P \text{ and } A\cap X\neq \emptyset \} \text{\, and\, } 
P_0=\{ A\in P: A\cap X= \emptyset\} \cup 
\{Y\cap \bigcup_{A\in P, A\cap X\neq \emptyset}A\}.$$
Since $ |P'|+(|P_0|-1)= |P|$, by Theorem~\ref{thm:Nash-Williams,Tutte}, 
we have 
$$e_{G[X\cup Y]}(P)\ge e_{G[X]}(P')+e_{G[Y]}(P_0)\ge m(|P'|-1)+m(|P_0|-1)= m(|P|-1).$$
Again, by applying Theorem~\ref{thm:Nash-Williams,Tutte}, the graph $G[X\cup Y]$ must be $m$-tree-connected.
}\end{proof}
The next observation presents a simple way for deducing tree-connectivity of a graph.
\begin{observ}\label{observ:deducing}
{Let $G$ be a graph and let $X\subseteq V(G)$. 
If $G[X]$ and $G/X$ are $m$-tree-connected, then $G$ itself is $m$-tree-connected.
}\end{observ}
\begin{proof}
{It is enough to apply the same argument as in the proof of Observation~\ref{observ:definition:tree-connected} by setting $Y=V(G)$.
Note that, by Theorem~\ref{thm:Nash-Williams,Tutte}, we again have $e_{G[Y]}(P_0)\ge m(|P_0|-1)$, since $G/X$ is $m$-tree-connected.
}\end{proof}
%
%
%
%
%
\subsection{Edge-density and the existence of non-trivial $m$-tree-connected components}
The following theorem shows that graphs with high edge-density must have $m^+$-tree-connected subgraphs.
\begin{thm}\label{thm:m+-subgraph}
{Every graph $G$ of order at least two containing at least $m(|V(G)|-1)+1$ edges has an
$m^+$-tree-connected subgraph with at least two vertices.
}\end{thm}
\begin{proof}
{The proof is by induction on $|V(G)|$. 
For $|V(G)|=2$, the proof is clear.
Assume $|V(G)|\ge 3$. Suppose the theorem is false. 
Since G is not $m^+$-tree-connected, there exits (by Theorem~\ref{thm:Nash-Williams,Tutte})
 an edge $e$ and a partition $P$ of $V(G)$ such that
 $e_{G_0}(P)< m(|P|-1)$, where $G_0=G-e$. Note that $e_{G}(P)\le m(|P|-1)$.
By induction hypothesis, for every $A\in P$ (since no part $A\in P$ has an $m^+$-tree-connected subgraph), 
we have $e_{G}(A)\le m(|A|-1)$ regardless of $|A|=1$ or not. 
Therefore, 
$$ m(|V(G)|-1) < |E(G)|= e_G(P)+\sum_{A\in P}e_G(A)\le m(|P|-1)+m \sum_{A\in P}(|A|-1)\le m(|V(G)|-1).$$
This result is a contradiction, as desired.
}\end{proof}
The following corollary 
 gives a sufficient condition for the existence of a non-trivial $m$-tree-connected subgraph.
\begin{cor}{\rm (\cite{Yao-Li-Lai-2010})}\label{lem:non-trivial-component}
{Every graph $G$ of order at least two containing at least $m(|V(G)|-1)$ edges has an
$m$-tree-connected subgraph with at least two vertices.
}\end{cor}
This condition can be improved a little when the existence of a non-trivial $m$-edge-connected subgraph is considered.
\begin{cor}{\rm (\cite[Section 2]{Thomassen-2013-fixedlength})}
{Every graph $G$ of order at least two containing at least $(m-1)(|V(G)|-1)+1$ edges has an
$m$-edge-connected subgraph with at least two vertices.
}\end{cor}
\begin{cor}\label{cor:spare-component}
{A graph $G$ is $m$-sparse if and only if its $m$-tree-connected components are $m$-sparse.
}\end{cor}
\begin{proof}
{Suppose, to the contrary, that $G$ is not $m$-sparse. Thus
 by Theorem~\ref{thm:m+-subgraph}, the graph $G$ has a non-trivial $m^+$-tree-connected subgraph $G[A]$.
Since $G[A]$ is not $m$-sparse, the $m$-tree-connected component of $G$ containing it is not $m$-sparse, which is a contradiction.
}\end{proof}
\begin{cor}\label{cor:maximum-degree:sparse}
{Every graph $M$ with maximum degree at most $m$ must be $m$-sparse.
}\end{cor}
\begin{proof}
{Suppose, to the contrary, that $M$ is not $m$-sparse. Thus
 by Theorem~\ref{thm:m+-subgraph}, the graph $M$ has a non-trivial $m^+$-tree-connected subgraph $M[A]$.
Since $M[A]$ is $(m+1)$-edge-connected, $M$ has minimum degree is at least $m+1$ which is a contradiction.
(A simpler proof can be obtained from the definition, because every non-trivial subgraph $M[A]$ has size at most $\frac{1}{2}m |A|$
which is less then or equal to $m(|A|-1)$).
}\end{proof}
\begin{cor}\label{cor:basictool:minimaly-tree-connected:M}
{Let $G$ be an $m$-tree-connected graph with an $m$-sparse factor $M$.
 If $G-e$ is not $m$-tree-connected for every $e\in E(G)\setminus E(M)$, then $G$ must be minimally $m$-tree-connected.
}\end{cor}
\begin{proof}
{Suppose, to the contrary, that $G$ is not $m$-sparse. Thus
 by Theorem~\ref{thm:m+-subgraph}, the graph $G$ has a non-trivial $m^+$-tree-connected subgraph $G[A]$.
Since $M[A]$ is $m$-sparse, there must be an edge $e\in E(G)\setminus E(M)$ with both ends in $A$.
Therefore, $G[A]-e$ must be $m$-tree-connected and so is $G-e$ which is a contradiction.
}\end{proof}
We shall here prove the following similar version of Corollary~\ref{cor:basictool:minimaly-tree-connected:M} based on Theorem~\ref{thm:m+-subgraph}. This result was proved in~\cite{Nagamochi-Ibaraki-1992, Nishizeki-Poljak-1994} when $M$ is the trivial factor. 
\begin{cor}
{Let $k_0$ and $k$ be two nonnegative integers, and
let $G$ be a $(k_0+k)$-edge-connected graph having a $k_0$-edge-connected factor $G_0$ and a $k$-sparse factor $M$.
If $G-e$ is not $(k+k_0)$-edge-connected for every $e\in E(G)\setminus E(G_0 \cup M )$, then 
 the supergraph $G\setminus E(G_0)$ of $M$ must be $k$-sparse.
}\end{cor}
\begin{proof}
{Suppose, to the contrary, that $H$ is not $k$-sparse, where $H=G\setminus E(G_0)$. Thus
 by Theorem~\ref{thm:m+-subgraph}, the graph $H$ has a non-trivial $k^+$-tree-connected subgraph $H[A]$.
Since $M[A]$ is $k$-sparse, there must be an edge $e\in E(H)\setminus E(M)$ with both ends in $A$.
Since $G-e$ is not $(k+k_0)$-edge-connected, there is a vertex set $X$ having exactly one end of $e$ such that $d_G(X)=k_0+k$.
Note that $A$ is not a subset of $X$ and $X \neq V(G_0)$. Since $H[A]$ is $(k+1)$-edge-connected, we must have
 $d_G(X)\ge d_{G_0}(X) + d_{H[A]}(X\cap A)\ge k_0 +(k+1)$, which is a contradiction.
}\end{proof}
\subsection{Sparsity of minimally $m$-tree-connected graphs}
Let $G$ be a graph satisfying $|E(G)| = m(|V(G)|-1)$.
It is known that $G$ is $m$-sparse if and only if $G$ is $m$-tree-connected. 
We shall below form a similar stronger version. 
\begin{thm}\label{thm:sparse-partition-connected}
{If $G$ is an $m$-sparse graph satisfying $|E(G)| = m(|V(G)|-1)$, 
then for every partition $P$ of $V(G)$,
$e_G(P)\ge m(|P|-1)+p_0$, where $p_0$ is the number of vertex sets $X\in P$ such that $G[X]$ is not $m$-tree-connected.
}\end{thm}
\begin{proof}
{Since $G$ is $m$-sparse, for every nonempty subset $A$ of $V(G)$, $e_G(A)\le m(|A|-1)$. 
Let $P_0$ be the set of all $A\in P$ such that $G[A]$ is not $m$-tree-connected.
Let $A\in P_0$.
By Theorem~\ref{thm:Nash-Williams,Tutte}, there exists a partition $P'$ of $A$ such that $e_{G[A]}(P')<m(|P'|-1)$. Thus $e_{G}(A)=e_{G[A]}(P')+\sum_{X\in P'}e_{G}(X)<m(|P'|-1)+ \sum_{X\in P'}m(|X|-1)=m(|A|-1)$. Therefore, 
$$e_G(P)=|E(G)|-\sum_{X\in P}e_G(X)\ge m(|V(G)|-1)-(\sum_{X\in P}m(|X|-1)-|P_0|)= m(|P|-1)+|P_0|.$$
This completes the proof.
}\end{proof}
The following corollary shows an application of Theorem~\ref{thm:sparse-partition-connected} that plays an essential role in this paper.
\begin{cor}\label{cor:Q:lowerbound}
{Let $F$ be an $m$-sparse graph and let $x,y\in V(G)$. 
Let $Q$ be a minimal $m$-tree-connected subgraph of $F$ containing $x$ and $y$. If $z\in V(Q)\setminus \{x,y\}$, then
$d_Q(z)\ge m+1$.
}\end{cor}
\begin{proof}
{Let $A=V(Q)\setminus \{z\}$.
By minimality of $Q$, the graph $Q[A]$ is not $m$-tree-connected.
In addition, $Q$ is $m$-sparse and $|E(Q)|=m(|V(Q)|-1)$.
Thus by Theorem~\ref{thm:sparse-partition-connected}, we must have
$d_Q(A)=e_Q(P)\ge m(|P|-1)+1= m+1$, where $P=\{A, \{z\}\}$.
Hence the proof is completed.
}\end{proof}
The next corollary is a very useful tool for finding a pair of edges such that replacing
them preserves sparsity of a given $m$-sparse factor. 
This tool for working with sparse graphs can be obtained using matroid theory; see~\cite{Edmonds-1970}. 
\begin{cor}\label{cor:xGy-exchange}
{Let $F$ be an $m$-sparse graph, let $x,y\in V(F)$, and let $e_{xy}$ be an edge joining $x$ and $y$ satisfying $e_{xy}\not \in E(F)$. 
If $Q$ is a minimal $m$-tree-connected subgraph of $F$ containing $x$ and $y$, then for every $e \in E(Q)$, the graph $F-e+e_{xy}$ remains $m$-sparse.
}\end{cor}
\begin{proof}
{Suppose, to the contrary, that $F-e+e_{xy}$ is not $m$-sparse.
 Thus there is a vertex subset $A$ of $V(F)$ containing $x$ and $y$ such that $e_F(A)\ge m(|A|-1)$ and $e\not \in E(F[A])$. 
This implies that $A\not \subseteq V(Q)$.
Since $F$ is $m$-sparse, we must have $e_F(A)= m(|A|-1)$.
Let $B=V(Q)$. 
Since $Q$ is $m$-tree-connected,
$e_F(A\cap B)\ge e_F(A)+e_F(B)-e_F(A\cup B)\ge
 m(|A|-1)+ m(|B|-1)-
 m(|A\cup B|-1)=m(|A\cap B|-1)$.
Therefore, $e_F(A\cap B)=m( |A\cap B|-1)$.
According to Theorem~\ref{thm:sparse-partition-connected}, the graph $F[A\cap B]$ must be $m$-tree-connected, which contradicts minimality of $Q$. Hence the proof is completed.
}\end{proof}
\begin{cor}\label{cor:xGy-exchange:tree-connected}
{Let $H$ be an $m$-tree-connected graph, let $x,y\in V(H)$, and let $e_{xy}$ be an edge joining $x$ and $y$ satisfying $e_{xy}\not \in E(H)$. If $Q$ is a minimal $m$-tree-connected subgraph of $H$ containing $x$ and $y$, then for every $e \in E(Q)$, the graph $H-e+e_{xy}$ remains $m$-tree-connected.
}\end{cor}
\begin{proof}
{By the assumption, $Q$ is minimally $m$-tree-connected and so it is $m$-sparse. Let $H'$ be a minimally $m$-tree-connected factor of $H$ containing $E(Q)$. By Corollary~\ref{cor:basictool:minimaly-tree-connected:M}, 
the graph $H'$ is minimally $m$-tree connected and so it is $m$-sparse.
Thus by Corollary~\ref{cor:xGy-exchange}, $H'-e+e_{xy}$ remains $m$-sparse. Since the size of $H'-e+e_{xy}$ is $m(|V(H')|-1)$, 
this graph must be $m$-tree-connected and so is $H-e+e_{xy}$.
}\end{proof}
\begin{cor}\label{lem:add-e}
{Let $F$ be an $m$-sparse graph, 
let $x,y\in V(F)$, and let $e_{xy}$ be an edge joining $x$ and $y$ satisfying $e_{xy}\not \in E(F)$. 
If $x$ and $y$ are in different $m$-tree-connected components of $F$, then $F+e_{xy}$ is still $m$-sparse.
}\end{cor}
\begin{proof}
{Suppose, to the contrary, that $F+e_{xy}$ is not $m$-sparse.
Thus there is a vertex set $A$ containing both ends of $e_{xy}$ such that $e_F(A)=m(|A|-1)$. 
From Theorem~\ref{thm:sparse-partition-connected}, we infer that $e_{F[A]}(P) \ge m(|P|-1)$ for all partitions $P$ of $A$,
 and so by Theorem~\ref{thm:Nash-Williams,Tutte} we find that $F[A]$ is $m$-tree-connected.
Hence both ends of $e_{xy}$ are in the same $m$-tree-connected component of $G$ which is a contradiction. 
}\end{proof}
\subsection{Some properties of tree-connectivity measures}
\label{sec:properties:measures}
The following theorem introduces an interesting property of tree-connectivity measures.
\begin{thm}\label{thm:maximum-partition}
{For every graph $G$, we have 
$$\Omega_m(G)=\max_{P\in \mathcal{A}}\{m|P|-e_G(P)\}\ge m,$$
where $\mathcal{A}$ is the set of all partitions $P$ of $V(G)$. In addition, the equality holds if and only if $G$ is $m$-tree-connected.
}\end{thm}
\begin{proof}
{Let $P$ be a partition of $V(G)$ with the maximum $m|P|-e_G(P)$.
 We may assume that $P$ has the minimum size.
We first claim that for every $X\in P$, $G[X]$ is $m$-tree-connected.
Otherwise, by Theorem~\ref{thm:Nash-Williams,Tutte}, there is a partition $P_X$ of $X$ such that $e_{G[X]}(P_X)<m (|P_X|-1)$.
Let $P'$ be the partition of $V(G)$ consisting of all parts of $P\setminus \{X\}$ and $P_X$.
Since $e_{G'}(P')=e_{G}(P)+e_{G[X]}(P_X)$, 
we must have $m|P'|-e_G(P')=m(|P|-1+|P_X|)-e_{G}(P)-e_{G[X]}(P_X)> m|P|-e_G(P) $, which is a contradiction. 
For every $X\in P$, let $X_0$ be the vertex set of the maximal $m$-tree-connected subgraph of $G$ containing $X$.
We claim that $X=X_0$.
Otherwise, by applying the first claim, there exists a partition $P_0$ of $X_0$ obtained from some parts of $P$ satisfying $|P_0|\ge 2$.
Let $P'=(P\setminus P_0) \cup \{X_0\}$. 
Note that $P'$ is a partition of $V(G)$ and $|P'| < |P|$.
Since $e_{G'}(P')=e_{G}(P)-e_{G[X_0]}(P_0)$, we must have 
$m|P'|-e_G(P')=m(|P|-|P_0|+1)-e_{G}(P)+e_{G[X_0]}(P_0)\ge m|P|-e_G(P) $, which is a contradiction to the minimality of $|P|$. 
Therefore, $P$ is the partition of $V(G)$ obtained from the $m$-tree-connected components of $G$. 
Thus $\Omega_m(G) =m|P|-e_G(P)$.
By applying Corollary~\ref{lem:non-trivial-component} to the contracted graph $G/P$, 
we must have $e_G(P)<m(|P|-1)$ provided that $|P|\ge 2$.
In addition, $e_G(P)=m(|P|-1)=0$ when $|P|=1$. 
 Thus $\Omega_m(G)=m$ if and only if $G$ is $m$-tree-connected.
Hence the proof is completed.
}\end{proof}
\begin{cor}\label{cor:maximum-partition}
{For every graph $G$, we have $\Omega_m(G) \ge m|V(G)|-|E(G)|$.
In addition, the equality holds if and only if $G$ is $m$-sparse.
}\end{cor}
\begin{proof}
{According to Theorem~\ref{thm:maximum-partition}, $\Omega_m(G) \ge m|P|-e_G(P)=m|V(G)|-|E(G)|$, where 
 $P$ is the trivial partition of $V(G)$ obtained from vertices of $G$. Now, we prove the second assertion.
Let $P$ be the partition of $V(G)$ obtained from $m$-tree-connected components of $G$. Obviously, $e_G(A)\ge m(|A|-1)$ fore all $A\in P$. In addition, by Corollary~\ref{cor:spare-component}, the equalities hold simultaneously if and only if $G$ is $m$-sparse.
Thus $|E(G)|=\sum_{A\in P}e_G(A)+e_G(P)\ge m|V(G)| -(m|P|-e_G(P))=m|V(G)|-\Omega_m(G)$.
Moreover, the equality holds if and only if $G$ is $m$-sparse.
Hence the assertion holds.
}\end{proof}
\begin{cor}\label{cor:comparing:e}
{Let $G$ be a graph. If $H$ is a factor of $G$, then $\Omega_m(G)\le \Omega_m(H)$.
}\end{cor}
\begin{proof}
{If we set $P$ to be the partition of $V(G)$ obtained from the $m$-tree-connected components of $G$, 
then by Theorem~\ref{thm:maximum-partition}, we must have 
$\Omega_m(G)=m|P|-e_{G}(P) \le m|P|-e_{H}(P)\le \Omega_m(H)$.
}\end{proof}
The following theorem describes a relationship between tree-connectivity measures of a graph.
\begin{thm}\label{thm:comparison}
{For every graph $G$, we have 
$$\omega(G)=\Omega_1(G)\le \frac{1}{2}\Omega_{2}(G)\le \frac{1}{3}\Omega_{3}(G) \le \cdots \le |V(G)|.$$
}\end{thm}
\begin{proof}
{Let $P$ and $P'$ be the partitions of $V(G)$ obtained from the $m$-tree-connected and 
$(m+1)$-tree-connected components of $G$, where $m\ge 1$.
For every $A\in P$, define $P_A$ to be the partition of $A$ obtained from the $(m+1)$-tree-connected components of $G[A]$.
It is easy to see that $P_A\subseteq P'$, and also $\cup_{A\in P}P_A=P'$.
Obviously, $e_{G[A]}(P_A) = (m+1)(|P_A|-1)=0$ when $|P_A|= 1$.
In addition, by applying Corollary~\ref{lem:non-trivial-component} to the graph $G[A]$, we must have
 $e_{G[A]}(P_A) < (m+1)(|P_A|-1)$ when $|P_A|\ge 2$.
Thus $$ e_G(P')- e_G(P)=
\sum_{A\in P}e_{G[A]}(P_A)\le 
(m+1)\sum_{A\in P}(|P_A|-1)=
(m+1)\big(|P'|-|P|\big).$$
Therefore, 
$$\frac{1}{m}\Omega_{m}(G)=|P|-\frac{1}{m}e_G(P)\le |P|-\frac{1}{m+1}e_G(P) \le |P'|-\frac{1}{m+1}e_G(P')=\frac{1}{m+1}\Omega_{m+1}(G).$$
This inequality completes the proof.
}\end{proof}
%
%
%
%
%
%
%
\section{Structures of tree-connected factors with the minimum total excess}
Here, we state the following fundamental theorem, which gives much information about $m$-tree-connected factors with the minimum total excess. This result develops some results in~\cite{Ellingham-Nam-Voss-2002, Enomoto-Ohnishi-Ota-2011,ClosedWalks, Win-1989} and it can also develop the main result of this paper in terms of total excess as the paper~\cite{Enomoto-Ohnishi-Ota-2011}.
\begin{thm}\label{thm:preliminary:structure}
{Let $G$ be an $m$-tree-connected graph and let $h$ be an integer-valued function on $V(G)$.
Let $M$ be a factor of $G$ satisfying $\Delta(M)\le m$.
 If $H$ is a minimally $m$-tree-connected factor of $G$ containing $M$ with the minimum total excess from $h$, 
then there exists a subset $S$ of $V(G)$ with the following properties:
\begin{enumerate}{
\item $\Omega_m(G\setminus S)=\Omega_m(H\setminus S)$.\label{Condition 2.1}
\item $S\supseteq \{v\in V(G):d_H(v)> h(v)\}$.\label{Condition 2.2}
\item For each vertex $v$ of $S$, $d_H(v)\ge h(v)$.\label{Condition 2.3}
}\end{enumerate}
}\end{thm}
\begin{proof}
{Note that $M$ is $m$-sparse (by Corollary~\ref{cor:maximum-degree:sparse}) and so $H$ is well-defined (by Corollary~\ref{cor:basictool:minimaly-tree-connected:M}).
Define $V_0=\emptyset $ and $V_1=\{v\in V(H): d_H(v)> h(v)\}$.
For any $S\subseteq V(G)$ and $u\in V(G)\setminus S$, 
let $C(S,u)$ be the vertex set of the $m$-tree-connected component of $H\setminus S$ containing $u$.
Let $\mathcal{A}(S, u)$ be the set of factors $H'$ of $G$ containing $M$ such that
\begin{enumerate}{
\item [(a)] $H'$ is minimally $m$-tree-connected.\label{Condition a}
\item [(b)] $d_{H'}(v)\le h(v)$ for all $v\in V(G)\setminus V_1$.\label{Condition b}
\item [(c)] if $e \in (E (H )\cup E (H')) \setminus (E (H )\cap E (H')$, then both ends of $e$ are contained in $C (S , u )$.\label{Condition c}
}\end{enumerate}
According to item (c), since $|E(H')|=|E(H)|=m(|V(H)|-1)$, we must have $e_{H'}(X)=e_H(X)$, where $X=C(S,u)$.
On the other hand, since $H[X]$ is $m$-sparse and $m$-tree-connected, $e_H(X)=m(|X|-1)$ which implies that $e_{H'}(X)=m(|X|-1)$. 
Thus the $m$-sparse graph $H'[X]$ must be minimally $m$-tree-connected with respect to
Theorems~\ref{thm:sparse-partition-connected} and~\ref{thm:Nash-Williams,Tutte}.
Now, for each integer $n$ with $n\ge 2$, recursively define $V_n$ as follows:
$$V_n=V_{n-1} \cup \{\, v\in V(G)\setminus V_{n-1} \colon \, d_{H^\prime }(v)\ge h(v) \text{\, for all\, }H^\prime \in \mathcal{A}(V_{n-1},v)\,\}.$$
 Now, we prove the following claim.
%
\vspace{2mm}\\
{\bf Claim.} 
Let $x$ and $y$ be two vertices in different m-tree-connected components of $H\setminus V_{n-1}$.
If $xy\in E(G)\setminus E(H)$, then $x\in V_{n}$ or $y\in V_{n}$.
\vspace{2mm}\\
{\bf Proof of Claim.} 
By induction on $n$.
 For $n=1$, the proof is clear. 
Assume that the claim is true for $n-1$.
Now we prove it for $n$.
Suppose, to the contrary, that vertices $x$ and $y$ are in different $m$-tree-connected components of $H\setminus V_{n-1}$,
$xy\in E(G)\setminus E(H)$, 
and $x,y\not \in V_{n}$. 
Let $X$ and $Y$ be the vertex sets of the $m$-tree-connected components of $H\setminus V_{n-1}$
containing $x$ and $y$, respectively. 
Since $x,y\not\in V_{n}$,
 there exist
 $H_x\in \mathcal{A}(V_{n-1},x)$ and
 $H_y\in \mathcal{A}(V_{n-1},y)$ with $d_{H_x}(x)< h(x)$ and $d_{H_y}(y)< h(y)$. 
By the induction hypothesis,
 $x$ and $y$ are in the same $m$-tree-connected component 
of $H\setminus V_{n-2}$ with the vertex set $Z$ so that $\{x,y\}\subseteq X\cup Y \subseteq Z$.
Obviously, $Z\cap V_{n-2}=\emptyset$.
Let $Q$ be a minimal $m$-tree-connected subgraph of $H[Z]$ containing $x$ and $y$. 

Notice that $Q$ includes at least
a vertex $z \in V_{n-1}\setminus V_{n-2}$ so that $d_H (z) \ge h(z)$. 
By Corollary~\ref{cor:Q:lowerbound}, we have $d_Q (z) > m\ge d_M(z)$ and so there is an edge $zz' \in E(Q)\setminus E(M)$.
By Corollary~\ref{cor:xGy-exchange}, the graph $H-zz'+xy$ is $m$-sparse. 
Since this graph contains $m(|V(H)|-1)$ edges, by Theorem~\ref{thm:sparse-partition-connected}, 
it must be $m$-tree-connected. 
Now, let $H'$ be the factor of $G$ with
 $$E(H')=E(H)-zz'+xy
-E(H[X])+E(H_x[X])
-E(H[Y])+E(H_y[Y]).$$
Recall that $H_x[X]$ and $H_y[Y]$ are minimally $m$-tree-connected. Thus $H'$ and $H$ have the same size $m(|V(H)|-1)$.
By applying Observation~\ref{observ:deducing} twice, one can easily check that $H'$ is minimally $m$-tree-connected. 
Since both of graphs $H_x$ and $H_y$ contain $E(M)$ and $zz'\not\in E(M)$, the graph $H'$ contains $E(M)$.
For each $v\in V(H')\setminus \{z\}$, we have
$$d_{H'}(v)= 
 \begin{cases}
d_{H_x}(v)-1,	&\text{if $v=z'\in X\setminus \{x\}$},\\
d_{H_y}(v)-1,	&\text{if $v=z' \in Y\setminus \{y\}$},\\
d_{H}(v)-1,	&\text{if $v=z'\not \in (X\cup Y)$},\\
d_{H_v}(v),	&\text{if $v=z'\in \{x,y\}$}.
\end {cases}
\quad \text{ and }\quad
 d_{H'}(v)= 
 \begin{cases}
d_{H_x}(v),	&\text{if $v\in X\setminus \{x,z'\}$};\\
d_{H_y}(v),	&\text{if $v\in Y\setminus \{y,z'\}$};\\
d_{H}(v),	&\text{if $v\notin X\cup Y\cup \{z,z'\}$};\\
d_{H_v}(v)+1,	&\text{if $v\in \{x,y\}\setminus \{z'\}$}.\\
\end {cases}$$
If $n\ge 3$, then it is not hard to see that $d_{H'}(z)=d_{H}(z)-1< h(z)$ and $H'$ lies in $ \mathcal{A}(V_{n-2}, z)$.
Since $z\in V_{n-1}\setminus V_{n-2}$, we arrive at a contradiction.
For the case $n=2$, 
since $z\in V_1$, it is easy to see that
$d_{H'}(z)=d_{H}(z)-1\ge h(z)$ and $te(H',h)< te(H,h)$, which is again a contradiction.
Hence the claim holds.

%
%
%
Obviously, there exists a positive integer $n$ such that and $V_1\subseteq \cdots\subseteq V_{n-1}=V_{n}$.
 Put $S=V_{n}$. 
Since $S\supseteq V_1$, Condition~\ref{Condition 2.2} clearly holds.
For each $v\in V_i\setminus V_{i-1}$ with $i\ge 2$, 
we have $H\in \mathcal{A}(V_{i-1},v)$ and so $d_H(v)\ge h(v)$. 
This establishes Condition~\ref{Condition 2.3}.
 Because $S=V_{n}$, 
the previous claim implies Condition~\ref{Condition 2.1} and completes the proof.
}\end{proof}
%
%
%
%
\section{Sufficient conditions depending on tree-connectivity measures}
Our aim in this subsection is to prove Theorem~\ref{intro:thm:sufficient}.
We begin with the following lemma that allows us to make the proof shorter. 
This lemma is an extension of Lemma 4.1 in \cite{ClosedWalks}.
\begin{lem}\label{lem:maximal-matching}
{Let $G$ be a graph with a factor $F$.
If a maximal factor $M$ of $F$ satisfying $\Delta(M)\le m$ can be extended to an $m$-tree-connected factor $T$, 
then $F$ itself can be extended to an $m$-tree-connected factor $H$ such that for each vertex $v$,
 $$d_H(v)\le d_T(v)+\max\{0, d_F(v)-m\}.$$
}\end{lem}
\begin{proof}
{Let $A=\{v\in V(G):d_M(v)< m\}$.
According to the maximality of $M$, the vertex set $A$ must be an independent set
of $F\setminus E(M)$. Otherwise, we can insert a new edge of $F\setminus E(M)$ into $M$ to expand it to a larger factor with maximum degree at most $m$ which is a
contradiction. 
Define $G_0=T\cup F$ and let $\mathcal{A}$ be the set of all $m$-tree-connected factors $T'_0$ of $G_0$ containing $M$ such that $d_{T'_0}(u)\le d_T(u)$ for all $u\in A$. 
 Note that $\mathcal{A}$ is nonempty, because $T\in \mathcal{A}$.
Consider $T_0\in \mathcal{A}$ with the maximum $|E(T_0)\cap E(F)|$ and let $H=T_0\cup F$. 
We claim that $H$ is the desired factor that we are looking for.
Let $v\in V(H)$.
If $d_{M}(v)= m$, then 
$$d_{H}(v)\le d_{G_0}(v) \le d_{T}(v)+d_F(v)-d_{M}(v)= d_T(v)+ d_F(v)-m.$$
So, suppose $v\in A$.
Define $F_0$ to be the factor of $F$ with $E(F_0)=E(F)\cap E(T_0)$.
To complete the proof, we are going to show that $d_{F_0}(v)\ge \min\{m, d_F(v)\}$. More precisely, this implies that
$$d_{H}(v) = d_{T_0}(v)+d_F(v)-d_{F_0}(v) \le d_{T}(v)+d_F(v)-\min\{m, d_F(v)\} = d_T(v)+\max\{0, d_F(v)-m\}.$$
Suppose, to the contrary, that $d_{F_0}(v)< \min\{m, d_F(v)\}$.
Pick $vx\in E(F)\setminus E(F_0)$ so that $vx\notin E(T_0)\cup E(M)$. 
Let $Q$ be a minimal $m$-tree-connected subgraph of $T_0$ containing $v$ and $x$. 
Since $Q$ is $m$-edge-connected, 
$d_Q (v) \ge m$.
On the other hand, $d_{F_0}(v) <m$.
Thus there exists an edge $vy \in E(Q)\setminus E(F_0)$ so that $vy\in E(T_0)\setminus E(F)$
(we might have $x = y$).
Define $T'_0=T_0-vy+vx$.
By Corollary~\ref{cor:xGy-exchange:tree-connected}, the graph $T'_0$
is still $m$-tree-connected. Note that $T'_0$ contains the edges of $M$, because $vy\notin E(M)$.
\begin{figure}[h]
{\centering
\includegraphics[scale=.95]{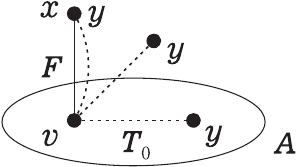}
\caption{An example showing all possibilities of the vertices $y$, $v$, and $x$.}
\label{Fig:12}
}\end{figure}
According to this construction, 
$d_{T'_0}(u)=d_{T_0}(u)$ for all $u\in V(G)\setminus \{x,y\}$. 
Moreover, $d_{T'_0}(x)=d_{T_0}(x)+1$ and $d_{T'_0}(y)=d_{T_0}(y)-1$ when $x\neq y$, and 
$d_{T'_0}(x)=d_{T_0}(x)$ when $x=y$.
Since $vx\in E(F)\setminus E(M)$ and $v\in A$, 
we must have $x\not \in A$.
Therefore, $d_{T_0'}(u)\le d_{T_0}(u)\le d_{T}(u)$ for all $u\in A$.
Since $|E(T_0')\cap E(F)| > |E(T_0)\cap E(F)|$, we derive a contradiction to the maximality of $T_0$, as desired.
}\end{proof}
The following lemma establishes a simple but important property of minimally $m$-tree-connected graphs.
\begin{lem}\label{lem:minimally,m-tree-connected}
{Let $H$ be a minimally $m$-tree-connected graph. If $ S\subseteq V(H)$, then
$$\Omega_m(H\setminus S)= \sum_{v\in S}(d_H(v)-m)+m-e_H(S).$$
}\end{lem}
\begin{proof}
{Let $P$ be the partition of $V(H)\setminus S$ obtained from the $m$-tree-connected components of $H\setminus S$.
Obviously, $e_H( P\cup \{\{v\}:v\in S\})=\sum_{v\in S}d_H(v)\,-e_H(S)+e_{H\setminus S}(P)$.
On the other hand, $e_H( P\cup \{\{v\}:v\in S\})=m(|P|+|S|-1)$, since
 $|E(H)|=m(|V(H)|-1)$ and for any $A\in P$, $e_H(A)= m(|A|-1)$.
Therefore, we must have
$\Omega_m(H\setminus S)=m|P|-e_{H\setminus S}(P)= \sum_{v\in S}(d_H(v)-m)+m-e_H(S)$. Hence the lemma holds.
}\end{proof}
The following theorem is essential in this section.
\begin{thm}\label{thm:sufficient}
Let $G$ be a graph with $X\subseteq V(G)$ and with a factor $F$. 
Let $f$ be a positive integer-valued function on $X$.
If for all $S\subseteq X$,
$$\Omega_m(G\setminus S)\le \sum_{v\in S}\big(f(v)-2m\big)+m+\Omega_m(G[ S]),$$
then $G$ has an $m$-tree-connected factor $H$ containing $F$ such that for each $v\in X$,
 $d_H(v)\le f(v)+\max\{0, d_F(v)-m\}$.
\end{thm}
\begin{proof}{Note that $G$ must automatically be $m$-tree-connected, because of $\Omega_m(G\setminus \emptyset) \le m$.
For each $v\in V(G)\setminus X$, define $f(v)=d_G(v)+1$. Choose a maximal factor $M$ of $F$ satisfying $\Delta(M)\le m$. 
Let $H$ be a minimally $m$-tree-connected factor of $G$ containing $M$ with the minimum total excess from $f$.
Define $S$ to be a subset of $V (G)$ with the properties described in Theorem~\ref{thm:preliminary:structure}.
If $v\in V(G)\setminus X$, then $d_{H}(v)\le d_G(v)<f(v)$. This implies that $S\subseteq X$.
By Lemma~\ref{lem:minimally,m-tree-connected} and Theorem~\ref{thm:preliminary:structure},
\begin{equation*}\label{main:eq:A:1}
\sum_{v\in S} f(v) +te(H,f)= \sum_{v\in S} d_H(v) = 
\Omega_m(H\setminus S)+m|S| -m+ e_{H}(S),
\end{equation*}
and hence
$$\sum_{v\in S} f(v) +te(H,f)= \Omega_m(G\setminus S)+m|S| -m+ e_{H}(S).
$$
Since $H[S]$ is $m$-sparse, by Corollaries~\ref{cor:maximum-partition} and~\ref{cor:comparing:e}, one can conclude that
$e_{H}(S) = m|S|-\Omega_m(H[S]) \le m|S|-\Omega_m(G[S])$.
By the assumption, we therefore have 
\begin{equation*}\label{main:eq:A:2}
te(H,f) \le \Omega_m(G\setminus S)- \sum_{v\in S}\big(f(v)-2m\big)-m-\Omega_m(G[S])\le 0.
\end{equation*}
Hence $te(H,f) = 0$. 
By Lemma~\ref{lem:maximal-matching}, the factor $F$ itself can be extended to an $m$-tree-connected connected factor $H$ 
such that for each vertex $v$,
$d_H(v)\le f(v)+\max\{0, d_F(v)-m\}$.
 Hence the theorem is proved.
}\end{proof}
The next corollary gives a sufficient condition, similar to toughness condition, that guarantees
the existence of a highly tree-connected factor with bounded maximum degree. 
\begin{cor}\label{cor:similar:toughness}
Let $G$ be an $m$-tree-connected graph with a factor $F$. 
Let $f$ be a positive integer-valued function on $V(G)$.
If for all $S\subseteq V(G)$,
$$\Omega_m(G\setminus S)\le \sum_{v\in S}\big(f(v)-2m\big)+2m,$$
then $G$ has an $m$-tree-connected factor $H$ containing $F$ such that for each vertex $v$,
 $d_H(v)\le f(v)+\max\{0,d_F(v)-m\}$.
\end{cor}
\begin{proof}
{Since G is $m$-tree-connected, it is obvious that $\Omega_m(G\setminus \emptyset) = m$.
Let $S$ be a nonempty subset of $V(G)$. Since $\Omega_m(G[S])\ge m $,
we must have $\Omega_m(G\setminus S)\le \sum_{v\in S}(f(v)-2m)+2m\le \sum_{v\in S}(f(v)-2m)+m+\Omega_m(G[ S])$.
Now, it is enough to apply Theorem~\ref{thm:sufficient}.
}\end{proof}
\begin{cor}\label{cor:X-independent}
{Let $G$ be a graph with an independent set $X\subseteq V(G)$ and with a factor $F$.
Let $f$ be a positive integer-valued function on $X$.
 If for all $S\subseteq X$, 
$$\Omega_m(G\setminus S)\le \sum_{v\in S}(f(v)-m)+m,$$ 
then every factor $F$ can be extended to an $m$-tree-connected factor $H$ such that for each $v\in X$,
$d_H(v)\le f(v)+\max\{0, d_F(v)-m\}.$
}\end{cor}
\begin{proof}
{Let $S$ be a subset of $X$. Since $X$ is an independent set, we must have $\Omega_m(G[S])=m|S|$ which implies that
$\Omega_m(G\setminus S)\le \sum_{v\in S}\big(f(v)-m\big)+m= \sum_{v\in S}\big(f(v)-2m\big)+m+\Omega_m(G[S])$.
Now, it is enough to apply Theorem~\ref{thm:sufficient}.
}\end{proof}
%
%
%
%
%
%
%
\section{Highly edge-connected graphs}
\label{sec:edge-connected-graphs}
\subsection{Edge-connected and tree-connected graphs}
\label{sec:edge-connected}
Highly edge-connected graphs are natural candidates for graphs satisfying the assumptions of Theorem~\ref{thm:sufficient}.
We examine them in this subsection, beginning with
 the following extended version of Lemma 4.9 in~\cite{ClosedWalks}.
\begin{lem}\label{lem:high-edge-connectivity:Omega}
{Let $G$ be a graph with $ S\subseteq V(G)$. Then
$$\Omega_m(G\setminus S)\le
 \begin{cases}
\sum_{v\in S}\frac{m (d_G(v)-2m)}{k}+\frac{2m}{k}\Omega_m(G[S]),	&\text{ if $G$ is $k$-edge-connected, $k\ge 2m$, and $S\neq \emptyset$};\\ 
 \sum_{v\in S}\big(\frac{m(d_G(v)-m)}{k}-m\big)+m+\frac{m}{k}\Omega_m(G[S]),	&\text{if $G$ is $k$-tree-connected and $k\ge m$}.
\end {cases}$$
}\end{lem}
\begin{proof}
{Let $P$ be the partition of $V(G)\setminus S$ obtained from the $m$-tree-connected components of $G\setminus S$. 
Obviously, we have
$$e_G( P\cup \{\{v\}:v\in S\})=\sum_{v\in S}d_G(v)\,-e_G(S)+e_{G\setminus S}(P).$$
If $G$ is $k$-edge-connected and $S\neq \emptyset$, then
 there are at least $k$ edges of $G$ with exactly one end in $C$, for any
$C\in P$. 
Thus
$e_G( P\cup \{\{v\}:v\in S\})\ge k|P|-e_{G\setminus S}(P)+e_G(S)$
and so if $k\ge 2m$, then 
$$\Omega_m(G\setminus S)= m|P|-e_{G\setminus S}(P)\le m|P|-\frac{2m}{k}e_{G\setminus S}(P)\le
 \sum_{v\in S}\frac{md_G(v)}{k}\,-\frac{2m}{k}e_G(S).$$
When $G$ is $k$-tree-connected, we have $e_G( P\cup \{\{v\}:v\in S\})\ge k(|P|+|S|-1)$
 and so if $k\ge m$, then
$$\Omega_m(G\setminus S)= m|P|-e_{G\setminus S}(P)\le m|P|-\frac{m}{k}e_{G\setminus S}(P)\le 
\sum_{v\in S}\big(\frac{m d_G(v)}{k}-m\big)+m-\frac{m}{k}e_G(S).$$
By Corollary~\ref{cor:maximum-partition}, since $e_G(S) \ge m|S|-\Omega_m(G[S])$, these inequalities complete the proof.
}\end{proof}
Now, we are ready to
generalize Theorems~\ref{intro:thm:edge-connected} as mentioned in the abstract.
\begin{thm}\label{thm:main:result}
{Let $G$ be a graph with $X\subseteq V(G)$. 
Then every factor $F$ can be extended to an $m$-tree-connected factor $H$
such that for each~$v\in X$, 
$$d_H(v)\le \max\{0, d_F(v) - m\}+
 \begin{cases}
 \big\lceil \frac{m}{k}(d_G(v)-2m)\big \rceil+2m,	&\text{if $G$ is $k$-edge-connected and $k\ge 2m$};\\ 
 \big\lceil \frac{m}{k}(d_G(v)-m)\big\rceil+m,	&\text{if $G$ is $k$-tree-connected and $k\ge m$};\\
 \big\lceil \frac{m}{k}d_G(v)\big \rceil+m,	&\text{if $G$ is $k$-edge-connected, $k\ge 2m$, and $X$ is independent};\\ 
 \big\lceil \frac{m}{k}d_G(v)\big\rceil,	&\text{if $G$ is $k$-tree-connected, $k\ge m$, and $X$ is independent}.
\end {cases}$$
Furthermore, for an arbitrary given vertex $u$, the upper bound can be reduced to
 $\lfloor \frac{m}{k}d_G(u)\rfloor + \max\{0, d_F(u) - m\}$.
}\end{thm}
\begin{proof}
{Let $S$ be a subset of $X$. If $S$ is empty, then $\Omega_m(G \setminus S) \le m$, 
since $G$ is $m$-tree-connected. Assume that $S$ is not empty.
Hence $\Omega_m(G[S]) \ge m$ by Theorem~\ref{thm:maximum-partition}.
If $G$ is $k$-edge-connected and $k\ge 2m$, then
by Lemma~\ref{lem:high-edge-connectivity:Omega}, 
we have 
$$\Omega_m(G\setminus S)\le \sum_{v\in S}\frac{m}{k}(d_G(v)-2m)+\frac{2m}{k}\Omega_m(G[S])<1+ 
 \sum_{v\in S}(f(v)-2m)+m+\Omega_m(G[S]),$$
where $f(u)= \lfloor \frac{m}{k}d_G(u)\rfloor$ and
 $f(v)= \lceil\frac{m(d_G(v)-2m)}{k}\rceil+2m$ for all $v\in V(G)\setminus \{u\}$.
If $G$ is $k$-tree-connected and $k\ge m$, then 
by Lemma~\ref{lem:high-edge-connectivity:Omega}, 
we also have 
$$\Omega_m(G\setminus S)\le \sum_{v\in S}(\frac{m}{k}(d_G(v)-m)-m)+m+\frac{m}{k}\Omega_m(G[S])< 1+
\sum_{v\in S}(f(v)-2m)+m+\Omega_m(G[S]),$$
where $f(u)= \lfloor \frac{m}{k}d_G(u)\rfloor$ and $f(v)= \lceil\frac{m(d_G(v)-m)}{k}\rceil+m$ for all $v\in V(G)\setminus \{u\}$.
Thus the first two assertions follow from Theorem~\ref{thm:sufficient}.
Now, suppose that $X$ is an independent set.
If $G$ is $k$-edge-connected and $k\ge 2m$, then
by Lemma~\ref{lem:high-edge-connectivity:Omega}, we have 
$$\Omega_m(G\setminus S)\le \sum_{v\in S}\frac{m}{k}d_G(v)<1+ 
\sum_{v\in S}(f(v)-m)+m,$$
where $f(u)= \lfloor \frac{m}{k}d_G(u)\rfloor$ and $f(v)= \lceil\frac{m d_G(v)}{k}\rceil+m$ for all $v\in X\setminus \{u\}$.
If $G$ is $k$-tree-connected and $k\ge m$, then 
by Lemma~\ref{lem:high-edge-connectivity:Omega}, we also have 
$$\Omega_m(G\setminus S)\le \sum_{v\in S}(\frac{m}{k}d_G(v)-m)+m< 1+
\sum_{v\in S}(f(v)-m)+m,$$
where $f(u)= \lfloor \frac{m}{k}d_G(u)\rfloor$ and $f(v)= \lceil\frac{m d_G(v)}{k}\rceil$ for all $v\in X\setminus \{u\}$.
Thus the second two assertions follow from Corollary~\ref{cor:X-independent}.
}\end{proof}
A generalization of Corollary~1 in \cite{SpanningTreeEulerian-2015} is given in the following corollary.
\begin{cor}
{ Every $(r-2m)$-edge-connected $r$-regular graph $G$ with $r\ge 4m$
 admits an $m$-tree-connected $\{m,m+1,\ldots,3m\}$-factor.
}\end{cor}
\begin{proof}
{Apply Theorem~\ref{thm:main:result} with $k=r-2m$.
}\end{proof}
\begin{cor}\label{cor:New}
{Every $2m$-edge-connected graph $G$ has an $m$-tree-connected factor $H$ 
such that for each vertex~$v$, 
$$d_H(v)\le \big\lceil \frac{d_G(v)}{2}\big\rceil+m.$$
Furthermore, for an arbitrary given vertex $u$, the upper bound can be reduced to $\lfloor \frac{d_G(u)}{2}\rfloor$.
}\end{cor}
\begin{proof}
{Apply Theorem~\ref{thm:main:result} with $k=2m$.
}\end{proof}
\subsection{Alternative proofs for $2m$-edge-connected graphs}
In the following, we shall give two simpler proofs for Corollary~\ref{cor:New} inspired by the proofs that introduced in~\cite{BangJensen-Thomasse-Yeo-2003, MR1621287, Thomassen-2008-P3} for the special case $m=1$.
For this purpose, we need some well-known results.
Note that the second one implicitly appeared in~\cite{Bensmail-Harutyunyan-Le-Thomasse-2019} for $m=2,6$.
Moreover, two interesting developments of this corollary are given in~\cite{AHO}.
\begin{thm}{\rm (Mader~\cite{Mader-1978}, see Section 3 in~\cite{Ok-Thomassen-2017})}\label{thm:Mader}
{Let $G$ be a $2m$-edge-connected graph with $z\in V(G)$.
If $d_G(z)\ge 2m+2$, then 
 there are two edges $xz$ and $yz$ incident to $z$ such that after removing them and inserting a new edge $xy$ for the case $x\neq y$, 
the resulting graph is still $2m$-edge-connected.
}\end{thm}

\begin{proofs}{
\noindent
\hspace*{-4mm}
\textbf{Corollary \ref{cor:New}.}
By induction on $\sum _{v\in V(G)}\max\{0, d_G(v)-2m-1\}$.
First, suppose that this summation is zero. This means that $\Delta(G)\le 2m+1$.
Since $G$ is $2m$-edge-connected, every vertex has degree $2m$ or $2m+1$.
Let $M\subseteq E_G(u)$ be an edge set of size 
$m$ or $m+1$ with respect to $d_G(u)=2m$ or $d_G(u)=2m+1$, 
where $E_G(u)$ denotes the set of edges of $G$ that are incident to $u$.
We claim that $G\setminus M$ is $m$-tree-connected and so the theorem obviously holds by setting $H=G\setminus M$.
Otherwise, Theorem~\ref{thm:Nash-Williams,Tutte} implies that there is a 
partition $P$ of $V(G)$
such that 
$m(|P|-1) > e_{G\setminus M}(P) \ge e_G(P)-|M|= \sum_{X\in P}d_G(X)/2-|M|\ge m|P|-|M|$.
This implies that the edges of $M$ join different parts of $P$,
$|M|=m+1$, $d_G(u)=2m+1$, and $d_G(X)=2m$ for all $X\in P$.
It is not hard to check that 
 $(E_G[U,\overline{U}]\cup E_G(u))\setminus M$ forms an edge cut of size $2m-1$ for $G$, which is contradiction, where $u\in U\in P$ and $E_G[U,\overline{U}]$ denotes the set of edges of $G$ with exactly one end in $U$.

Now, suppose that there is a vertex $z$ with $d_G(z)\ge 2m+2$. 
By Theorem~\ref{thm:Mader}, there are two edges $xz$ and $yz$ incident to $z$ such that 
after removing them, and inserting a new edge $xy$ for the case $x\neq y$, 
the resulting graph $G'$ is still $2m$-edge-connected.
By the induction hypothesis, the graph $G'$ has a factor $H'$ containing $m$ edge-disjoint spanning trees 
$T_1,\ldots, T_m$ such that $d_{H'}(u)\le \lfloor d_{G'}(u)/2\rfloor$ and for each vertex $v$ with $v\neq u$, $d_{H'}(v)\le \lceil d_{G'}(v)/2\rceil+m$. 
If $xy \not \in E(T_1\cup \cdots \cup T_m)$, then the theorem clearly holds. 
Thus we may assume that $xy\in E(T_1)$ and $z$ and $x$ lie in the same component of $T_1-xy$. 
Define $T'_1= T_1-xy+yz$. 
It is easy to see that $T'_1$ is connected and $T'_1\cup T_2\cup \cdots \cup T_m$ 
is the desired factor of $G$ that we are looking for.
}\end{proofs}
Before stating the third proof, let us establish the following lemma and state the next two well-known results.
\begin{lem}\label{lem:l0}
{Let $G$ be a graph with a factor $M$ satisfying $\Delta(M) \le m$. 
If $G$ can be decomposed into an $m$-tree-connected factor $H$ and a factor $F$ having an orientation such that for each vertex $v$, $d^+_F(v)\ge l_0(v)$, then $G$ can also be decomposed into an $m$-tree-connected factor $H'$ containing $M$ and a factor $F'$ having an orientation such that for each vertex $v$, $d^+_{F'}(v)\ge l_0(v)$, where $l_0$ is a nonnegative integer-valued function on $V(G)$.
}\end{lem}
\begin{proof}
{Decompose $G$ into a minimally $m$-tree-connected factor $H$ and a factor $F$ having an orientation such that for each vertex $v$, $d^+_F(v)\ge l_0(v)$. Consider the pair $(H,F)$ with the maximum $|E(H) \cap E(M)|$. 
We claim that $H$ contains the edges of $M$. Suppose, to the contrary, that there is an edge $vx\in E(M) \cap E(F)$.
We may assume that $vx$ is directed from $v$ to $x$ in $F$. Let $Q$ be a minimal $m$-tree-connected factor of $G$ containing $v$ and $x$. Since $d_Q(v)\ge m> d_{H\cap M}(v)$, there is an edge $vy\in E(Q)\setminus E(M)$.
We define $H_0=H+vx-vy$ and $F_0=F-vx+vy$, and we orient the edge $vy$ from $v$ to $y$ in $F_0$. 
Obviously, for each vertex $u$, we still have $d^+_{F_0}(u)=d^+_F(u)$.
Moreover, by Corollary~\ref{cor:xGy-exchange:tree-connected}, the graph $H_0$ is still $m$-tree-connected.
Thus the new pair $(H_0, F_0)$ has the desired properties while $|E(H_0) \cap E(M)| >|E(H) \cap E(M)|$, which is a contradiction.
Hence the proof is completed.
}\end{proof}
\begin{thm}{\rm (Nash-Williams~\cite{MR0118684}, see Theorem 2.1 in~\cite{BangJensen-Thomasse-Yeo-2003})}\label{thm:Nash-Williams}
{Every $2m$-edge-connected graph $G$ has an $m$-arc-strong orientation such that for each vertex $v$, 
$\lfloor d_G(v)/2\rfloor\le d^+_G(v) \le \lceil d_G(v)/2\rceil$.
}\end{thm}
\begin{thm}{\rm(Edmonds~\cite{Edmonds-1973})}\label{thm:Edmonds}
{Let $G$ be a directed graph with $u\in V(G)$. 
If $d^-_G(X)\ge m$ for all $X\subsetneq V(G)$ with $u\in X$, then
 $G$ has a spanning subdigraph $H$ such that its underlying graph is $m$-tree-connected, $d^+_{H}(u)=0$, and $d^+_{H}(v)=m$
 for all $v\in V(G)\setminus \{u\}$, where $d^-_G(X)$ denotes the number of incoming edges to~$X$ in $G$.
}\end{thm}
Now, we are in a position to provide the third proof of Corollary~\ref{cor:New} by proving the following stronger version.
\begin{thm}\label{thm:New}
{Let $G$ be a graph with a factor $M$ satisfying $\Delta(M)\le m$. If $G$ is $2m$-edge-connected, then it can be decomposed into an $m$-tree-connected factor $H$ containing $M$ and a factor $F$ having an orientation such that for each vertex $v$, 
$$d^+_F(v)\ge \lfloor \frac{d_G(v)}{2}\rfloor-m.$$
Furthermore, for an arbitrary given vertex $u$ the lower bound can be increased to $\lceil \frac{d_G(u)}{2}\rceil$.
}\end{thm}
\begin{proof}{
Consider an $m$-arc-strong orientation for $G$ with the properties stated in Theorem~\ref{thm:Nash-Williams}. 
We may assume that the out-degree of $u$ is equal to $\lceil d_G(u)/2\rceil$; otherwise, we reverse the orientation of $G$.
Take
$H$ to be a spanning subdigraph of $G$ with the properties stated in Theorem~\ref{thm:Edmonds}.
For each vertex $v$, we have
$d_H(v)=d^-_H(v)+d^+_H(v)\le d^-_G(v)+d^+_H(v) \le \lceil d_G(v)/2\rceil+m$.
In particular, 
$d_H(u)\le d^-_G(u)+d^+_H(u) \le \lfloor d_G(u)/2\rfloor$.
These imply that for each vertex $v$, $d^+_F(v)\ge \lfloor \frac{d_G(v)}{2}\rfloor-m$, and $d^+_F(u)\ge \lceil \frac{d_G(u)}{2}\rceil$, where $F$ is the complement of $H$ in $G$.
Now, it is enough to apply Lemma~\ref{lem:l0} to complete the proof.
}\end{proof}
%
%
%
%
%
%
%
%
%
%
%
%
%

\section{Tough enough graphs}
\label{sec:tough-graph}
\subsection{The existence of $m$-tree-connected $[m,2m+1]$-factors}
\label{subsec:toughenough}

As we have already shown in Theorem~\ref{thm:comparison}, $m$-strongly tough enough graphs are tough enough.
In this subsection, we shall prove the converse statement and examine
 tough enough graphs for Corollary~\ref{cor:similar:toughness}.
For this purpose, we need the following two lemmas.
\begin{lem}\label{lem:optimized}
{Let $G$ be a graph and let $\varepsilon$ be a real number with $0 \le \varepsilon \le 1$. If $S$ is a subset of $V(G)$ 
with the maximum $\Omega_m(G\setminus S)-\varepsilon|S|$ and with the maximal $|S|$, 
then every component of $G\setminus S$ is $m$-tree-connected or has maximum degree at most $m$.
}\end{lem}
\begin{proof}
{Let $v$ be an arbitrary vertex of $G\setminus S$ and define $S'=S\cup \{v\}$.
If $v$ is contained in a non-trivial $m$-tree-connected component of $G\setminus S$ with vertex set $X$, then 
it is not difficult to check that
$\Omega_m(G\setminus S')=\Omega_m(G\setminus S)-m+\Omega_m(G[X\setminus v])+d$, 
where $d$ denotes the number of edges incident to $v$ having one end in $V(G)\setminus (X\cup S)$.
Thus by Theorem~\ref{thm:maximum-partition},
$\Omega_m(G\setminus S')-\varepsilon|S'|\ge \Omega_m(G\setminus S)-\varepsilon|S| -\varepsilon+d$.
According to the assumption, we must have $d<\varepsilon$ and so $d=0$.
Therefore, every non-trivial $m$-tree-connected component of $G\setminus S$ is also a component of it.
If $v$ is a trivial $m$-tree-connected component of $G\setminus S$, then 
it is not difficult to check that
$\Omega_m(G\setminus S') = \Omega_m(G\setminus S)-m+d_{G\setminus S}(v)$.
Thus $\Omega_m(G\setminus S')-\varepsilon|S'|= \Omega_m(G\setminus S)-\varepsilon|S| -\varepsilon-m+d_{G\setminus S}(v)$.
According to the assumption, we must have $d_{G\setminus S}(v)< m+\varepsilon$ and so $d_{G\setminus S}(v)\le m$.
This implies that if a component of $G\setminus S$ is not $m$-tree-connected, then it has maximum degree at most $m$.
This completes the proof.
}\end{proof}
\begin{lem}{\rm (\cite{Isolated})}\label{lem:base:independent-subset}
{Let $H$ be a graph. If $\varphi$ is a nonnegative real function on $V(H)$,
 then there is an independent subset $I$ of $V(H)$ such that 
$$
\sum_{v\in V(H)} \varphi(v)\le \sum_{v\in I}\varphi(v) (d_H(v)+1) .
$$
}\end{lem}
Now, we are ready to prove the main result of this section.
\begin{thm}\label{thm:iso:tree-connected}
{Let $G$ be a graph and let $\varepsilon$ and $c$ be two real numbers satisfying $0\le \varepsilon\le 1/m $ and $1 \le c$. 
If for all $S\subseteq V(G)$, 
$$ \omega(G\setminus S)+\small {\frac{m+1}{2}}\, iso(G\setminus S) \le \varepsilon |S|+c,$$
then for all $S\subseteq V(G)$, $$\frac{1}{m} \Omega_m(G\setminus S) \le \varepsilon |S|+c.$$
}\end{thm}
\begin{proof}{
Let $S\subseteq V(G)$ with the properties described in Lemma~\ref{lem:optimized}.
Denote by $\sigma$ the number of non-trivial components of $G\setminus S$ which are $m$-tree-connected.
Let $C$ be the induced subgraph of $G$ consisting of the vertices of trivial $m$-tree-connected components of $G\setminus S$.
By Lemma~\ref{lem:base:independent-subset}, there is an independent set $I$ of $C$ 
such that 
$$\sum_{v\in V(C)}(1+\varepsilon-\frac{d_C(v)}{2m})\le 
\sum_{v\in I}(1+\varepsilon-\frac{d_C(v)}{2m})(d_C(v)+1).
$$
If $V(C)$ is empty, then we must automatically have 
$\frac{1}{m}\Omega_m(G\setminus S)= \omega (G\setminus S) \le \varepsilon |S|+c$.
We may therefore assume that $V(C)$ is nonempty and so is $I$. 
Since $\Delta(C)\le m$ and $ 0\le \varepsilon \le 1/m$, we must have 
\begin{equation*}
{\sum_{v\in I}(1+\varepsilon- \frac{d_C(v)}{2m})(d_C(v)+1)\le 
\frac{1}{2m}\sum_{v\in I}(2m+2m\varepsilon-m)(m+1)\le (\frac{m+3}{2}+\varepsilon)|I|,
}\end{equation*}
which implies that 
\begin{equation}\label{eq:thm:iso:tree-connected:1}
{\sum_{v\in V(C)}(1-\frac{d_C(v)}{2m}) \le (\frac{m+3}{2}+\varepsilon)|I|-\varepsilon|V(C)| .
}\end{equation}
Let $S'=S\cup (V(C)\setminus I)$ so that 
$\omega(G\setminus S') =\sigma+|I|$ and $iso(G\setminus S') =|I|$.
Thus by the assumption, 
$$
\sigma+\frac{m+3}{2}|I|= 
\omega(G\setminus S')+ \frac{m+1}{2}\, iso(G\setminus S')\le
 \varepsilon |S'|+c = \varepsilon (|S|+|V(C)|-|I|)+c,
$$
which implies that 
\begin{equation}\label{eq:thm:iso:tree-connected:2}
{\sigma+(\frac{m+3}{2}+\varepsilon)|I|-\varepsilon |V(C)| 
\le \varepsilon |S|+c.
}\end{equation}
Therefore, 
Inequations (\ref{eq:thm:iso:tree-connected:1}) and (\ref{eq:thm:iso:tree-connected:2}) can conclude that
$$\frac{1}{m}\Omega_m(G\setminus S)= 
 \sigma+\sum_{v\in V(C)}(1-\frac{d_C(v)}{2m})\le
\varepsilon|S|+c.$$
Hence the theorem holds.
}\end{proof}
The following corollary says that tough enough graphs with sufficiently large
order are also $m$-strongly tough enough.
\begin{cor}\label{cor:tough-enough}
{Let $G$ be a graph and let $t$ be a real number with $t\ge 1$. 
If $G$ is $2m^2t$-tough and $|V(G)|\ge 2m^2 t$, 
then $G$ is also $m$-strongly $t$-tough.
}\end{cor}
\begin{proof}
{We may assume that $m\ge 2$. 
Since $G$ is $\frac{m+3}{2}(m+1) t$-tough, for every $S\subseteq V(G)$, 
$ \omega(G\setminus S)+\small {\frac{m+1}{2}}\, iso(G\setminus S) \le \frac{1}{(m+1) t}|S|+1$ provided that
 $\omega(G\setminus S) \ge 2$ or $iso(G\setminus S) = 0$.
In addition, if $\omega(G\setminus S) = iso(G\setminus S) =1$, then $|S|=|V(G)|-1$, and so 
$ \omega(G\setminus S)+\small {\frac{m+1}{2}}\, iso(G\setminus S)=
\frac{m+3}{2} \le \frac{1}{ (m+1) t}|S|+1$.
Thus by Theorem~\ref{thm:iso:tree-connected}, 
for every $S\subseteq V(G)$,
 $ \frac{1}{m} \Omega_m(G\setminus S) \le \frac{1}{(m+1)t}|S|+1$.
Therefore, if $\Omega_m(G\setminus S) >m$, 
then we must have $|S|\ge \frac{(m+1)t}{m}$ which implies that 
$\frac{1}{m}\Omega_m(G\setminus S) \le \frac{1}{(m+1)t} |S|+1 \le \frac{1}{t} |S|$.
This means that $G$ is $m$-strongly $t$-tough.
}\end{proof}
The following result gives a sufficient toughness-type condition for a graph to have an $m$-tree-connected factor with maximum degree at most $2m+1$.
\begin{cor}\label{cor:tough+iso}
{Let $G$ be a graph with a factor $M$ satisfying $\Delta(M)\le m$. If for all $S\subseteq V(G)$, 
$$ \omega(G\setminus S)+\small {\frac{m+1}{2}}\, iso(G\setminus S) \le \frac{1}{m}|S|+1,$$
then $G$ admits an $m$-tree-connected factor $H$ containing $M$ such that for each vertex $v$, $d_H(v)\le 2m+1$,
and also $d_H(u)\le m+1$ for an arbitrary given vertex $u$.
}\end{cor}
\begin{proof}
{By applying Theorem~\ref{thm:iso:tree-connected} with setting $\varepsilon=1/m$ and $c =1$, we must have $ \Omega_m(G\setminus S) \le |S|+m$ for all $S\subseteq V(G)$. Hence the assertion follows from Corollary~\ref{cor:similar:toughness} with setting $f(u)=m+1$ and $f(v)=2m+1$ for each $v\in V(G)\setminus \{u\}$. Note that $G$ must automatically be $m$-tree-connected, because of $\Omega_m(G\setminus \emptyset)\le m$.
}\end{proof}
\begin{cor}\label{cor:tough-m2:verify}
{Every $\frac{1}{2}m(m+3)$-tough graph $G$ of order at least $2m$ has an $m$-tree-connected factor $H$ satisfying $\Delta(H)\le 2m+1$.
}\end{cor}
\begin{proof}
{If $|V(G)|\le \frac{1}{2}m(m+3)$, then $G$ must be complete and the proof is straightforward.
Assume that $|V(G)|> \frac{1}{2}m(m+3)$.
Let $S$ be a subset of $V(G)$.
If $\omega(G\setminus S)\ge 2$, then 
$ \omega(G\setminus S)+\small {\frac{m+1}{2}}\, iso(G\setminus S) \le 
\frac{m+3}{2}\omega(G\setminus S)\le \frac{1}{m}|S|$.
If $\omega(G\setminus S)\le 1$ and $iso(G\setminus S) = 0$, then
 $\omega(G\setminus S)+\small {\frac{m+1}{2}}\, iso(G\setminus S)\le 1$.
If $\omega(G\setminus S)= iso(G\setminus S) = 1$, then $S=|V(G)|-1\ge \frac{1}{2}m(m+3)-1$
and so
$ \omega(G\setminus S)+\small {\frac{m+1}{2}}\, iso(G\setminus S) = 
\frac{m+3}{2} \le \frac{1}{m}|S|+1$.
Now, it is enough to apply Corollary~\ref{cor:tough+iso}.
}\end{proof}
%
%
%
\subsection{Sharpness: $(m-\varepsilon)$-tough graphs with no $m$-tree-connected $[m,2m+1]$-factors}
 Our aim in this subsection is to present a family of tough graphs with no $m$-tree-connected $[m,2m+1]$-factors and show that the coefficient $1/m$ of Corollary~\ref{cor:tough+iso} is sharp. For our purpose, we first need to establish the following simple lemma.
\begin{lem}\label{lem:m-epsilon}
{Let $m$ be an integer with $m\ge 2$. For every $\varepsilon \in (0,1)$, there is a simple graph $G$ satisfying
 $e_G(P)\ge (m-\varepsilon)(|P|-1)$ for every partition $P$ of $V(G)$, while $G$ is not $m$-tree-connected. 
}\end{lem}
\begin{proof}
{Let $G_0$ be an essentially $(4m-2)$-edge-connected $2m$-regular simple graph $G_0$ of order at least $ 1/\varepsilon +1$; 
a graph is called {\it essentially $\lambda$-edge-connected},
 if all edges of any edge cut of size strictly less than $\lambda$ are incident to a common vertex.
(For example, the graph with vertices $v_1, \ldots, v_n$ and edges $v_iv_{i+j}$ (modulo $n$) where $j\in \{1,\ldots, m\}$ and $n$ is a sufficiently large integer). Since $m\ge 2$, the graph $G_0$ is essentially $(2m+2)$-edge-connected.
Take $M$ to be a set of $m+1$ edges of $G_0$ which are not incident to a common vertex. 
We claim that $G=G_0\setminus M$ is the desired graph.
Since $|E(G)|=m(|V(G)|-1)-1$, the graph $G$ is not $m$-tree-connected. 
Let $P$ be a partition of $V(G)$ of size at least two.
First, assume that $|P|= 2$.
If there exists a set $X\in P$ satisfying $|X| = 1$, then by the assumption, $e_G(P)\ge 2m-(|M|-1)\ge m$.
Otherwise, $e_G(P)\ge 2m+2-|M|> m$.
Thus, in both cases, $e_G(P)\ge m(|P|-1)$.
Now, assume that $|P|\ge 3$.
If there exists a set $X\in P$ satisfying $|X| = 2$, then $d_{G_0}(X)\ge 2m+2$ 
and hence $e_G(P)\ge e_{G_0}(P)-|M|\ge 
\sum_{A\in P}d_G(A)/2-|M| \ge m|P|+1-(m+1)= m(|P|-1)$.
Otherwise, $|P|=|V(G)|$ and hence
$e_G(P)=|E(G)|=m(|V(G)|-1)-1\ge (m-\varepsilon)(|P|-1)$.
This completes the proof.
}\end{proof}
The following theorem shows that the coefficient $1/m$ in Corollary~\ref{cor:tough+iso} is sharp and cannot be improved 
even by arbitrarily increasing the coefficient of $iso(G\setminus S)$ or the upper bound on the maximum degree.
\def\Ks {\text{$R_s$}}
\begin{thm}\label{thm:exam:iso:tree-connected}
{Let $m$ an integer with $m\ge 2$, and let $\Delta$, $k$ and $n$ be arbitrary positive integers. 
If $\varepsilon\in (0,1)$, then there exist infinitely many $k$-connected simple graphs $G$ having no $m$-tree-connected factor with maximum degree at most $\Delta$, while for all $S\subseteq V(G)$ satisfying $|S|\ge k$, 
$$ \omega(G\setminus S)+n\, iso(G\setminus S) \le (\frac{1}{m}+\varepsilon)(|S|-k)+1.$$
}\end{thm}
\begin{proof}
{Let $s$ and $p$ be two arbitrary integers satisfying $s\ge k$ and $p > \Delta s$. 
By Lemma~\ref{lem:m-epsilon}, there exists a non-$m$-tree-connected simple graph $H'$ satisfying
 $e_{H'}(P)\ge m'(|P|-1)$ for every partition $P$ of $V(G)$, where $m'=1 / (1/m+\varepsilon) < m$.
We replace each vertex of $H'$
by a copy of the complete graph $K_{n_0}$ such that every vertex in the new graph is adjacent to at most one edge of $H'$, where 
 $n_0$ is a fixed positive integer satisfying $n_0\ge m' (n+1)|V(H')|+1$. 
We call the resulting graph $H$.
For every integer $i$ with $1\le i\le p$, let $H_i$ be a copy of the graph $H$ and
 let $U_i$ be the set of all vertices of a complete subgraph of $H_i$ corresponding to $K_{n_0}$.
First, add all possible edges between all vertices of $U_i$ and $U_j$, where $i, j\in \{1,\ldots, p\}$. 
Next, add a new complete graph $\Ks$ of order $s$ and join all its vertices to all vertices of every graph $H_i$.
We call the resulting graph $G$; see Figure~\ref{Fig:partition-connected}.

Obviously, $G$ must be $s$-connected. 
Suppose, to the contrary, that $G$ has an $m$-tree-connected factor $T$ with maximum degree at most $\Delta$.
According to the construction of $G$, since $H_i$ is not $m$-tree-connected, 
there must be at least one edge of $T$ having one end in $V(\Ks)$ the other one in $V(H_i)\setminus U_i$. 
Thus there are at least $p$ edges of $T$ with exactly one end in $V(\Ks)$.
Therefore, $T$ contains a vertex in $V(\Ks)$ with degree at least $p/s> \Delta$ which is a contradiction.
\begin{figure}[h]
{\centering
\includegraphics[scale=.95]{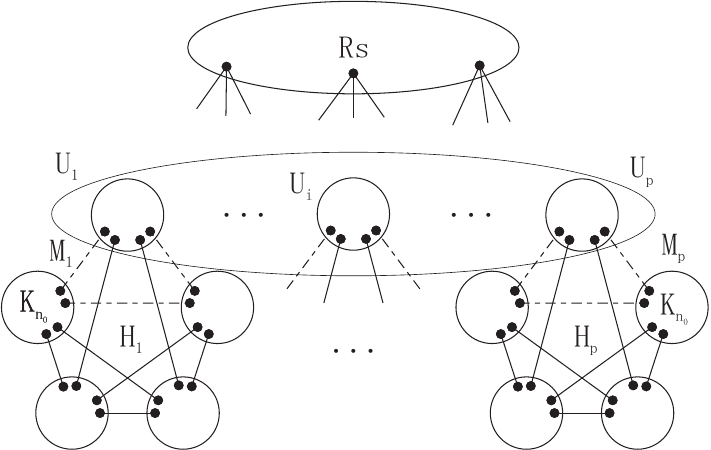}
\caption{Graph structure for the special case $m=2$.}
\label{Fig:partition-connected}
}\end{figure}

Let $S\subseteq V(G)$ with $|S|\ge k$. 
If $\omega(G\setminus S)\le 1$ and $iso(G)= 0$, then
$\omega(G\setminus S)+n\, iso(G\setminus S)=1 \le \frac{1}{m'}(|S|-k)+1$.
If $\omega(G\setminus S)=iso(G)= 1$, then $|S|= |V(G)|-1\ge n_0+k$ and so
$\omega(G\setminus S)+n\, iso(G\setminus S)=(n+1) \le \frac{1}{m'}(|S|-k)$.
 We may assume that $\omega(G\setminus S)\ge 2$. 
Obviously, $S$ must contain all vertices of $\Ks$.
For every integer $i$ with $1 \le i \le p$, let $S_i=S\cap V(H_i)$, 
Thus $|S|=s+\sum_{1\le i\le p}|S_i|$.
Define $\omega_i$ to be the number of components of $H_i\setminus S_i$ having no vertices of $U_i$.
According to the construction of $G$, we must have 
$\omega(G\setminus S)\le \sum_{1\le i\le p}\omega_i+1$.
In addition, if $|S_i|\le n_0-2$ then $ iso(H_i\setminus S_i) =0$. 
Let $P_i$ be the partition of $V(H')$ obtained from the components of $H_i\setminus S_i$.
According to the construction of $H'$, we must have
$$|S_i| \ge e_{H'}(P_i) \ge m'(|P_i|-1) = m' (\omega(H_i\setminus S_i)-1)=m'\omega_i,$$
which implies that 
$\omega_i +n\, iso(H_i\setminus S_i) =\omega_i \le \frac{1}{m'}|S_i|.$
If $|S_i|\ge n_0-1$, then we must have 
 $|S_i|\ge m' (n+1)|V(H_i)|$, which again implies that
$\omega_i +n\, iso(H_i\setminus S) \le (n+1)|V(H)|\le \frac{1}{m'}|S_i|$.
Therefore, one can conclude that
$$
\frac{|S|-k}{\omega(G\setminus S)-1+n\, iso(G\setminus S)}\ge 
\frac{\sum_{1\le i\le p}|S_i|}{\sum_{1\le i\le p}(\omega_i+n\, iso(H_i\setminus S_i))}
\ge m'.$$
These inequalities complete the proof.
}\end{proof}
%
%
\subsection{An application to highly connected star-free simple graphs}

Matthews and Summer (1984)~\cite{Matthews-Summer-1984} showed that every $k$-connected $K_{1,3}$-free 
simple graph is $\frac{k}{2}$-tough. This result is generalized to $k$-connected $K_{1,n}$-free simple graphs 
 by Jackson and Wormald~(1990)~\cite{Jackson-Wormald-1990} and Chen and Schelp (1995)~\cite{Chen-Schelp-1995} independently. 
In the following lemma, we provide a stronger version for their result for $K_{1,f}$-free simple graphs.
\begin{lem}\label{lem:tough+iso:star-free graphs}
{Let $G$ be a simple graph and let $f$ be an integer-valued function on $V(G)$ with $f\ge 2$.
 If $G$ is $k$-connected and $K_{1,f}$-free, then for every $S\subseteq V(G)$, 
$$ \omega(G\setminus S)+\sum _{v\in I(G\setminus S)}(\frac{1}{k}d_G(v)-1) \le \max\{1, \frac{1}{k}\sum_{v\in S}(f(v)-1)\}.$$
}\end{lem}
\begin{proof}
{We repeat the proof of Theorem 4.2 in~\cite{Jackson-Wormald-1990} with some modifications. 
Let $S$ be a subset of $V(G)$.
Since $G$ is $k$-connected, every component of $G\setminus S$
is joined to at least $k$ vertices in $S$ when $ \omega(G\setminus S) \ge 2$. 
In addition, every component of $G\setminus S$ consisting of a single vertex $v$ is joined to at least $d_G(v)$ vertices in $S$.
Since $G$ is $K_{1,f}$-free, every vertex of $S$ is joined to at most $f(v)-1$ components of $G\setminus S$.
Thus $k(\omega(G\setminus S)-iso(G\setminus S))+\sum _{v\in I(G\setminus S)}d_G(v)\le \sum_{v\in S}(f(v)-1)$ provided that 
 $ \omega(G\setminus S) \ge 2$ or $ \omega(G\setminus S) = iso (G\setminus S) = 1$.
Therefore, for all vertex sets $S$, $\omega(G\setminus S)+\sum _{v\in I(G\setminus S)}(\frac{1}{k}d_G(v)-1) \le 
\max\{1, \frac{1}{k}\sum_{v\in S}(f(v)-1)\}$ regardless of $\omega(G\setminus S) \le 1$ and $iso (G\setminus S)=0$ or not. Hence the assertion holds.
}\end{proof}
The following result is an application of Lemma~\ref{lem:tough+iso:star-free graphs} and Corollary~\ref{cor:tough+iso}. 
\begin{thm}\label{thm:star-free-graphs}
{Every $m(n-1)$-connected $K_{1,n}$-free simple graph $G$ with $n\ge 3$ admits an $m$-tree-connected factor $H$ satisfying $\Delta(H) \le 2m+1$, provided that $\delta(G)\ge \frac{1}{2}m(m+3)(n-1)$.
}\end{thm}
\begin{proof}
{By Lemma~\ref{lem:tough+iso:star-free graphs}, for every $S\subseteq V(G)$, we must have 
$$ \omega(G\setminus S)+(\frac{\delta(G)}{m(n-1)}-1) iso(G\setminus S)\le 
 \omega(G\setminus S)+\frac{m+1}{2}iso(G\setminus S)\le \max\{1, \frac{n-1}{m(n-1)}|S|\} \le \frac{1}{m}|S|+1.$$
Thus by Corollary~\ref{cor:tough+iso}, the graph $G$ has an $m$-tree-connected factor $H$ satisfying the theorem.
}\end{proof}

\section{Acknowledgements}
The author would like to thank anonymous referees for their helpful comments.
%
%
%
%
%
%
%
%
%
%
%

\end{document}